\documentclass[12pt, reqno]{amsart}
\usepackage{amsmath, amsthm, amscd, amsfonts, amssymb, graphicx, color}
\usepackage[bookmarksnumbered, colorlinks, plainpages]{hyperref}

\newtheorem{theorem}{Theorem}[section]
\newtheorem{lemma}[theorem]{Lemma}

\newtheorem{corollary}[theorem]{Corollary}
\theoremstyle{definition}

\theoremstyle{remark}

\numberwithin{equation}{section}
\begin{document}
\setcounter{page}{1}

\title[Differences of weighted differentiation composition operators]{Differences of weighted differentiation composition operators from Bloch-type space to weighted-type space}

\author[Y.X. Liang ]{Yu-Xia Liang}

\address{$^{1}$ School of Mathematical Sciences, Tianjin Normal University, Tianjin 300387, P.R. China.}
\email{\textcolor[rgb]{0.00,0.00,0.84}{liangyx1986@126.com}}

%\address{$^{*}$ Department of Mathematics, Tianjin University, Tianjin 300072, P.R. China;}
%\email{\textcolor[rgb]{0.00,0.00,0.84}{zehuazhoumath@aliyun.com}}

%\dedicatory{This paper is dedicated to Professor ABCD}


 \thanks{This work was supported  by  the Doctoral Fund of Tianjin Normal University (Grant Nos. 52XB1514). }

%\date{Received: xxx,2012; Revised: yyyyyy; Accepted:May 16, 2013.}
%\newline \indent $^{*}$ Corresponding author}

\begin{abstract}
 We found several new equivalent characterizations for the boundedness of the differences of weighted differentiation composition operators from Bloch-type space to weighted-type space. Especially, we estimated its essential norm in terms of the $n$-th power of the induced analytic self-maps on the unit disk, which can provide a new and simple compactness criterion.
\end{abstract} \maketitle

\section{Introduction and preliminaries}
Denote $\mathbb{N}_0$ the set of all nonnegative integers. In the sequel, the notations $A\approx B,\;A\preceq B,\;A\succeq B$ mean that there maybe different  positive constants $C$ such that $B/C \leq A \leq CB,\; A \leq CB,\;  C B\leq A.$ Let $H(\mathbb{D})$ be the space of all holomorphic functions on $\mathbb{D}$ and $S(\mathbb{D})$ the collection of all holomorphic self-maps on $\mathbb{D},$ where $\mathbb{D}$ is the unit disk in the complex plane $\mathbb{C}.$  Given a continuous linear operator $T$ on a Banach space $X$,  its essential norm  is the distance from the operator $T$ to compact operators on $X$, that is, $\|T\|_{e}=\inf\{\|T-K\|: K\; \mbox{is compact}\}.$ It's trivial that  $\|T\|_{e}=0$ if and only if $T$ is compact, see, e.g. \cite{GM} and their references therein.

For $a\in \mathbb{D}$, let $\varphi_a$ be the automorphism of $\mathbb{D}$ exchanging $0$ for $a,$ that is, $\varphi_a(z)=(a-z)/(1-\bar{a}z).$ For $z,w\in \mathbb{D},$ the pseudo-hyperbolic distance between $z$ and $w$ is given by $$\rho(z,w)=|\varphi_w(z)|=\left|\frac{z-w}{1-\bar{w}z}\right|.$$ Immediately, given $\varphi_1,\varphi_2\in S(\mathbb{D}),$ we denote $\rho(z)=\rho(\varphi_1(z),\varphi_2(z))$ for simplicity.

  For $0<\alpha<\infty,$ an $f \in H(\mathbb{D})$ is said to be
in the Bloch-type space $\mathcal{B}^\alpha,$ or $\alpha-$Bloch space, if
\begin{eqnarray*}\|f\|_{\mathcal{B}^\alpha}=|f(0)|+\sup\limits_{z\in \mathbb{D}} (1-|z|^2)^\alpha |f'(z)|<\infty. \end{eqnarray*} As we all know,  $\mathcal{B}^\alpha$ is a Banach space endowed with the norm $\|f\|_{\mathcal{B}^\alpha},$ and the little Bloch-type space $\mathcal{B}_0^\alpha$ is the closure of polynomials in $\mathcal{B}^\alpha$, see,e.g. \cite{FZ,HL,MZ,MZ1,Zhu1}.  In particular, $\mathcal{B}^\alpha=\mathcal{B},$ the classical Bloch space for $\alpha=1$; if $0<\alpha<1,$ $\mathcal{B}^\alpha=Lip_{1-\alpha},$ the analytic Lipschitz space which consists of all $f\in H(\mathbb{D})$ satisfying
 \begin{eqnarray*} |f(z)-f(w)| \leq C|z-w|^{1-\alpha},\end{eqnarray*} for some constant $C>0$ and all $z,w\in \mathbb{D};$
 when $\alpha>1,$ $\mathcal{B}^{\alpha}=H_{\alpha-1}^\infty,$ the $\alpha-1$ weighted-type space of analytic functions that contains all $f\in H(\mathbb{D})$ satisfying
 \begin{eqnarray*}\sup\limits_{z\in \mathbb{D}}(1-|z|^2)^{\alpha-1}|f(z)|<\infty. \end{eqnarray*}
 More generally, let $v$ be a strictly positive continuous and bounded function (weight) on $\mathbb{D}$. The weighted-type space $H_v^\infty$ is defined to be the collection of all functions $f\in H(\mathbb{D})$ that satisfy $$\|f\|_v=\sup\limits_{z\in \mathbb{D}} v(z)|f(z)|<\infty,$$  provided we identify that differ by a constant, and then $H_v^\infty$ is a Banach space under the norm $\|.\|_v$, see, e.g. \cite{DO,HKLRS} and the references therein.

Given $\varphi\in S(\mathbb{D})$ and $u\in H(\mathbb{D})$, the weighted
composition operator $uC_{\varphi}$  is defined by $$uC_{\varphi}(f) =u
\cdot(f\circ \varphi)\;\;\mbox{for $f\in H(\mathbb{D}).$ }$$ As for $u\equiv 1,$ the
weighted composition operator is the usual
composition operator, denote by $C_\varphi$, see \cite{CM}.  When $\varphi=id$   the
identity map, the operator $uC_{id}$ is called
multiplication operator $M_u$. Let $D=D^1$ be the differentiation  operator, i.e., $Df=f'$ for $f\in H(\mathbb{D}).$ More generally, given an integer $m\in \mathbb{N}_0,$ we can further define the operator $D^mf=f^{(m)}$ for $f\in H(\mathbb{D}).$ Now, the weighted differentiation composition operator, denoted by $D_{\varphi,u}^m,$ is given as $$(D_{\varphi,u}^mf)(z)=u(z)f^{(m)}(\varphi(z)),\;\mbox{for $f\in H(\mathbb{D}).$ }$$ In fact, the operator $D_{\varphi,u}^m$ can degenerate to many classical operators, such as $D_{\varphi,id}^0=C_\varphi$ with $u=id$ and $m=0$; $D_{\varphi,u}^{0}=uC_\varphi$ with $m=0$; $D_{\varphi,id}^1=C_\varphi D$ with $u=id$ and $m=1$, and $D_{\varphi,u}^1=uC_\varphi D$ with $m=1$.

In 2009, interest has arisen to characterize the properties of composition operator $C_\varphi$  on Bloch-type spaces in terms of the $n$-th power of the analytic self-map $\varphi$ of the open unit disk $\mathbb{D}$. More clearly,  Wulan, Zheng and Zhu \cite{WZZ} obtained a new result about the
compactness of the composition operator on the Bloch space. It's said that $C_\varphi$ is compact on the Bloch space $\mathcal{B}$ if and only if  $\lim\limits_{n\rightarrow \infty} \|\varphi ^n\|_{\mathcal{B}}=0,$ where $\varphi^n$ means the $n$-th power of $\varphi.$ As regards to Bloch-type spaces, Zhao \cite{Zhao} obtained that  $\|C_\varphi\|_{e,\mathcal{B}^\alpha\rightarrow \mathcal{B}^\beta}\approx \limsup\limits_{n\rightarrow \infty} n^{\alpha-1}\|\varphi^n\|_{\beta}$ for $0<\alpha,\beta<\infty.$   As far as we know that the composition operator is a typical bounded operator on the classical Bloch space $\mathcal{B},$ while the differentiation operators are typically unbounded on many Banach spaces of holomorphic functions. Especially, giving the new equivalent characterizations for the boundedness and compactness of  weighted differentiation composition operator $D_{\varphi,u}^m$ are interesting thing, which can unify many classical operators as above. There has been some work on composition and differentiation operators between holomorphic spaces,  and the interested readers can refer to \cite{LZ2,LZ1,LZ3, St1,WW} and their references therein   on much of the developments in the theory of new characterizations. As far as we know, there has been no new similar descriptions for differences of operators. Hence the characterizations for differences of classical operators  by  the $n$-th power of the induced analytic self-maps are in desired need of response. In this paper, we will try our best to characterize the boundedness and compactness of the operator $D_{\varphi_1,u_1}^m-D_{\varphi_2,u_2}:\;\mathcal{B}_\alpha\rightarrow H_v^\infty$. The paper is organized as follows: we found several characterizations for the boundedness of $ D_{\varphi_1,u_1}^m-D_{\varphi_2,u_2}^m:  \mathcal{B}^\alpha\rightarrow H_v^\infty$ in section 2; and then the compactness of $ D_{\varphi_1,u_1}^m-D_{\varphi_2,u_2}^m:  \mathcal{B}^\alpha\rightarrow H_v^\infty$ was considered in section 3; finally, some corollaries were presented in section 4.

\section{ The boundedness of $ D_{\varphi_1,u_1}^m-D_{\varphi_2,u_2}^m:  \mathcal{B}^\alpha\rightarrow H_v^\infty$}

In this section, we will give several equivalent characterizations for the boundedness of $ D_{\varphi_1,u_1}^m-D_{\varphi_2,u_2}^m:  \mathcal{B}^\alpha\rightarrow H_v^\infty$. For $a\in \mathbb{D},$ we define the following two families test functions:
\begin{eqnarray}f_a(z)=\int_0^z\int_0^{t_m}\cdots \int_0^{t_2} \frac{(1-|a|^2)^\alpha}{(1-\bar{a}t_1)^{2\alpha+m-1}}dt_1dt_2\cdots dt_m,\label{f_a} \end{eqnarray}
\begin{eqnarray} g_a(z)=\int_0^z\int_0^{t_m}\cdots \int_0^{t_2} \frac{(1-|a|^2)^\alpha}{(1-\bar{a}t_1)^{2\alpha+m-1}}\cdot \frac{a-t_1}{1-\bar{a}t_1} dt_1dt_2\cdots dt_m.\label{g_a}\end{eqnarray}

Due to the fact $f\in \mathcal{B}^\alpha$ if and only if $\|f\|_{\mathcal{B}^\alpha}\approx \sup\limits_{z\in \mathbb{D}}(1-|z|^2)^{\alpha+m-1}|f^{(m)}(z)|<\infty.$ It's obvious that the following equality holds \begin{eqnarray*}&&\|g_a\|_{\mathcal{B}^\alpha}\preceq \|f_a\|_{\mathcal{B}^\alpha}\approx \sup\limits_{z\in \mathbb{D}}(1-|z|^2)^{\alpha+m-1}|f_a^{(m)}(z)|\nonumber\\&&=\sup\limits_{z\in \mathbb{D}}(1-|z|^2)^{\alpha+m-1} \frac{(1-|a|^2)^\alpha}{|1-\bar{a}z|^{2\alpha+m-1}}<\infty.\end{eqnarray*}Moreover, by the direct computations,  it yields that
\begin{eqnarray} f_a^{(m)}(z)= \frac{(1-|a|^2)^\alpha}{(1-\bar{a}z)^{2\alpha+m-1}} \;\; \mbox{and}\;\; g_a^{(m)}(z)=\frac{(1-|a|^2)^\alpha}{(1-\bar{a}z)^{2\alpha+m-1}} \cdot \frac{a-z}{1-\bar{a}{z}}.\label{fgam}\end{eqnarray}
For our further use, we denote two notations \begin{eqnarray*} \mathcal{T}_{\alpha+m-1}^{\varphi_1} (v u_1)(z)=\frac{v(z)u_1(z)}{(1-|\varphi_1(z)|^2)^{\alpha+m-1}},\;  \mathcal{T}_{\alpha+m-1}^{\varphi_2} (v u_2)(z)=\frac{v(z)u_2(z)}{(1-|\varphi_2(z)|^2)^{\alpha+m-1}}; \end{eqnarray*}

In order to estimate the differences, we prove an estimate for $  | (1-|z|^2)^{\alpha+m-1}f^{(m)}(z)-(1-|w|^2)^{\alpha+m-1} f^{(m)}(w)|$ for $f\in \mathcal{B}^\alpha$ and $z,w\in \mathbb{D}.$
\begin{lemma}\label{lem di} Let $0<\alpha<\infty$. Then for each $f\in \mathcal{B}^\alpha,$ it holds that
$$  | (1-|z|^2)^{\alpha+m-1}f^{(m)}(z)-(1-|w|^2)^{\alpha+m-1} f^{(m)}(w)|\leq C \|f \|_{\mathcal{B}^\alpha}\rho(z,w)$$  for all $z,w \in \mathbb{D}.$ \end{lemma}
\begin{proof}For $f\in \mathcal{B}^\alpha,$ it follows that $\sup\limits_{z\in \mathbb{D}}(1-|z|^2)^\alpha|f'(z)|\approx \sup\limits_{z\in \mathbb{D}}(1-|z|^2)^{\alpha+m-1}|f^{(m)}(z)|<\infty.$ That is to say $f^{(m)}\in H_{\alpha+m-1}^\infty,$ and moreover $\|f^{(m)}\|_{H_{\alpha+m-1}^\infty}\preceq \|f\|_{\mathcal{B}^\alpha}$.  By \cite[Lemma 3.2]{DO}, it yields that % a  $|v(z)f(z)-v(w)f(w)|\leq C\|f\|_v \rho(z,w)$ for all $f\in H_v^\infty$ with $v$ a radial weight satisfying condition (L1) \cite{}, we deduce that

\begin{eqnarray*} && | (1-|z|^2)^{\alpha+m-1}f^{(m)}(z)-(1-|w|^2)^{\alpha+m-1} f^{(m)}(w)|\\&&\preceq \|f^{(m)}\|_{H_{\alpha+m-1}^\infty} \rho(z,w)\preceq \|f \|_{\mathcal{B}^\alpha}\rho(z,w).\end{eqnarray*} This ends the proof.
  \end{proof}
\begin{lemma}\label{lemma TFG} Let $m\in \mathbb{N}_0,$ $0<\alpha<\infty$ and $v$ be a weight. Suppose $u_1, u_2\in H(\mathbb{D}),$ $\varphi_1, \varphi_2 \in S(\mathbb{D})$. Then the following three inequalities hold,

\begin{eqnarray}&&(i)\;\;\sup\limits_{z\in \mathbb{D}}\left|\mathcal{T}_{\alpha+m-1}^{\varphi_1}(vu_1)(z)\right| \rho(z) \nonumber\\&& \leq\sup\limits_{a\in \mathbb{D}} \|(D_{\varphi_1,u_1}^m-D_{\varphi_2,u_2}^m) f_a \|_v +\sup\limits_{a\in \mathbb{D}}\|(D_{\varphi_1,u_1}^m-D_{\varphi_2,u_2}^m) g_a\|_v. \\&& (ii)\;\; \sup\limits_{z\in \mathbb{D}}\left|\mathcal{T}_{\alpha+m-1}^{\varphi_2}(vu_2)(z)\right| \rho(z) \nonumber \\&&\leq  \sup\limits_{a\in \mathbb{D}} \|(D_{\varphi_1,u_1}^m-D_{\varphi_2,u_2}^m) f_a\|_v +\sup\limits_{a\in \mathbb{D}}\|(D_{\varphi_1,u_1}^m-D_{\varphi_2,u_2}^m) g_a\|_v.\\&& (iii)\;\;\sup\limits_{z\in \mathbb{D}}\left|\mathcal{T}_{\alpha+m-1}^{\varphi_1}(vu_1)(z)-\mathcal{T}_{\alpha+m-1}^{\varphi_2}
(vu_2)(z)\right|\nonumber\\&&  \leq  \sup\limits_{a\in \mathbb{D}} \|(D_{\varphi_1,u_1}^m-D_{\varphi_2,u_2}^m) f_a\|_v +\sup\limits_{a\in \mathbb{D}}\|(D_{\varphi_1,u_1}^m-D_{\varphi_2,u_2}^m) g_a\|_v.  \end{eqnarray} That is, \begin{eqnarray*}&& \sup\limits_{z\in \mathbb{D}}\left|\mathcal{T}_{\alpha+m-1}^{\varphi_1}(vu_1)(z)\right|\rho(z)+\sup\limits_{z\in \mathbb{D}}\left|\mathcal{T}_{\alpha+m-1}^{\varphi_2}(vu_2)(z)\right| \rho(z)\\&&+\sup\limits_{z\in \mathbb{D}}\left|\mathcal{T}_{\alpha+m-1}^{\varphi_1}(vu_1)(z)-\mathcal{T}_{\alpha+m-1}^{\varphi_2}
(vu_2)(z)\right|\\&& \leq  \sup\limits_{a\in \mathbb{D}} \|(D_{\varphi_1,u_1}^m-D_{\varphi_2,u_2}^m) f_a\|_v +\sup\limits_{a\in \mathbb{D}}\|(D_{\varphi_1,u_1}^m-D_{\varphi_2,u_2}^m) g_a\|_v.\end{eqnarray*} \end{lemma}
\begin{proof} For any $z\in \mathbb{D},$ we obtain that
\begin{eqnarray}&&\|(D_{\varphi_1,u_1}^m-D_{\varphi_2,u_2}^m) f_{\varphi_1(z)}\|_v \geq v(z)|(D_{\varphi_1,u_1}^m-D_{\varphi_2,u_2}^m) f_{\varphi_1(z)}(z)|\nonumber\\&& = v(z)|u_1(z)f_{\varphi_1(z)}^{(m)}(\varphi_1(z))-u_2(z)f_{\varphi_1(z)}^{(m)}(\varphi_2(z))| \nonumber\\&&=v(z)\left| \frac{u_1(z)}{(1-|\varphi_1(z)|^2)^{\alpha+m-1}}- \frac{u_2(z)(1-|\varphi_1(z)|^2)^\alpha}{(1-\overline{\varphi_1(z)}\varphi_2(z))^{2\alpha+m-1}}\right|
\quad\quad\quad\label{D1}\\&&\geq \left|\mathcal{T}_{\alpha+m-1}^{\varphi_1}(vu_1)(z)-
\frac{(1-|\varphi_1(z)|^2)^\alpha(1-|\varphi_2(z)|^2)^{\alpha+m-1}}{(1-\overline{\varphi_1(z)}\varphi_2(z))
^{2\alpha+m-1}} \mathcal{T}_{\alpha+m-1}^{\varphi_2}(vu_2)(z)\right|\nonumber\\&&\geq \left|\mathcal{T}_{\alpha+m-1}^{\varphi_1}(vu_1)(z)\right|-
\frac{(1-|\varphi_1(z)|^2)^\alpha(1-|\varphi_2(z)|^2)^{\alpha+m-1}}{|1-\overline{\varphi_1(z)}\varphi_2(z)|
^{2\alpha+m-1}}\left| \mathcal{T}_{\alpha+m-1}^{\varphi_2}(vu_2)(z)\right|.\nonumber\end{eqnarray} Similarly, it turns out that

\begin{eqnarray*} &&\|(D_{\varphi_1,u_1}^m-D_{\varphi_2,u_2}^m) g_{\varphi_1(z)}\|_v\\ &&\geq v(z)|u_1(z)g_{\varphi_1(z)}^{(m)}(\varphi_1(z))-u_2(z)g_{\varphi_1(z)}^{(m)}(\varphi_2(z))| \nonumber\\&&=v(z)| u_2(z)| \frac{(1-|\varphi_1(z)|^2)^\alpha}{|1-\overline{\varphi_1(z)}\varphi_2(z)|^{2\alpha+m-1}} \rho(z)\\&&=   \frac{(1-|\varphi_1(z)|^2)^\alpha(1-|\varphi_2(z)|^2)^{\alpha+m-1}}{|1-\overline{\varphi_1(z)}
\varphi_2(z)|^{2\alpha+m-1}}
|\mathcal{T}_{\alpha+m-1}^{\varphi_2}(vu_2)(z)|\rho(z).
\end{eqnarray*} On the one hand, we employ the above two inequalities to obtain that
\begin{eqnarray}&&\left|\mathcal{T}_{\alpha+m-1}^{\varphi_1}(vu_1)(z)\right| \rho(z)  \leq \|(D_{\varphi_1,u_1}^m-D_{\varphi_2,u_2}^m) f_{\varphi_1(z)}\|_v  \rho(z)\nonumber \\  &&+   \frac{(1-|\varphi_1(z)|^2)^\alpha(1-|\varphi_2(z)|^2)^{\alpha+m-1}}{|1-\overline{\varphi_1(z)}
\varphi_2(z)|^{2\alpha+m-1}}
|\mathcal{T}_{\alpha+m-1}^{\varphi_2}(vu_2)(z)|\rho(z) \nonumber\\&& \leq  \|(D_{\varphi_1,u_1}^m-D_{\varphi_2,u_2}^m) f_{\varphi_1(z)}\|_v +\|(D_{\varphi_1,u_1}^m-D_{\varphi_2,u_2}^m) g_{\varphi_1(z)}\|_v, \label{T1}\end{eqnarray} where the last inequality follows from $\rho(z)\leq 1.$ Analogously,  we deduce that
\begin{eqnarray} \left|\mathcal{T}_{\alpha+m-1}^{\varphi_2}(vu_2)(z)\right| \rho(z)&& \leq  \|(D_{\varphi_1,u_1}^m-D_{\varphi_2,u_2}^m) f_{\varphi_2(z)}\|_v \nonumber\\&&+\|(D_{\varphi_1,u_1}^m-D_{\varphi_2,u_2}^m) g_{\varphi_2(z)}\|_v.\;\;\;\label{T2}\end{eqnarray}
From \eqref{T1} and \eqref{T2}, we arrive at
\begin{eqnarray}&&(i)\;\sup\limits_{z\in \mathbb{D}}\left|\mathcal{T}_{\alpha+m-1}^{\varphi_1}(vu_1)(z)\right| \rho(z)\nonumber\\&&\leq \sup\limits_{z\in \mathbb{D}}\left( \|(D_{\varphi_1,u_1}^m-D_{\varphi_2,u_2}^m) f_{\varphi_1(z)}\|_v +\|(D_{\varphi_1,u_1}^m-D_{\varphi_2,u_2}^m) g_{\varphi_1(z)}\|_v\right)\nonumber\\&& \leq\sup\limits_{a\in \mathbb{D}}\left( \|(D_{\varphi_1,u_1}^m-D_{\varphi_2,u_2}^m) f_a\|_v +\|(D_{\varphi_1,u_1}^m-D_{\varphi_2,u_2}^m) g_a\|_v\right).\\&&(ii)\; \sup\limits_{z\in \mathbb{D}}\left|\mathcal{T}_{\alpha+m-1}^{\varphi_2}(vu_2)(z)\right| \rho(z) \nonumber \\&&\leq  \sup\limits_{a\in \mathbb{D}}\left(\|(D_{\varphi_1,u_1}^m-D_{\varphi_2,u_2}^m) f_a\|_v +\|(D_{\varphi_1,u_1}^m-D_{\varphi_2,u_2}^m) g_a\|_v\right).\end{eqnarray}

On the other hand, we change \eqref{D1} into
\begin{eqnarray}&&\|(D_{\varphi_1,u_1}^m-D_{\varphi_2,u_2}^m) f_{\varphi_1(z)}\|_v\nonumber\\&&=v(z)\left| \frac{u_1(z)}{(1-|\varphi_1(z)|^2)^{\alpha+m-1}}- \frac{u_2(z)(1-|\varphi_1(z)|^2)^\alpha}{(1-\overline{\varphi_1(z)}\varphi_2(z))^{2\alpha+m-1}}\right|
\nonumber\\&&\geq \left|\mathcal{T}_{\alpha+m-1}^{\varphi_1}(vu_1)(z)
-\mathcal{T}_{\alpha+m-1}^{\varphi_2}(vu_2)(z)\right| - \frac{v(z)u_2(z)}{(1-|\varphi_2(z)|^2)^{\alpha+m-1}}\nonumber\\&& \cdot\left| (1-|\varphi_1(z)|^2)^{\alpha+m-1} f_{\varphi_1(z)}^{(m)}(\varphi_1(z))-(1-|\varphi_2(z)|^2)^{\alpha+m-1}f_{\varphi_1(z)}^{(m)}
(\varphi_2(z))\right| \nonumber\\&&= \left|\mathcal{T}_{\alpha+m-1}^{\varphi_1}(vu_1)(z)
-\mathcal{T}_{\alpha+m-1}^{\varphi_2}(vu_2)(z)\right|-\left|\mathcal{T}_{\alpha+m-1}^{\varphi_2}
(vu_2)(z)\right|\nonumber\\&& \cdot\left| (1-|\varphi_1(z)|^2)^{\alpha+m-1} f_{\varphi_1(z)}^{(m)}(\varphi_1(z))-(1-|\varphi_2(z)|^2)^{\alpha+m-1}f_{\varphi_1(z)}^{(m)}
(\varphi_2(z))\right| \nonumber\\&&\succeq \left|\mathcal{T}_{\alpha+m-1}^{\varphi_1}(vu_1)(z)
-\mathcal{T}_{\alpha+m-1}^{\varphi_2}(vu_2)(z)\right|-\left|\mathcal{T}_{\alpha+m-1}^{\varphi_2}
(vu_2)(z)\right|\rho(z), \nonumber\end{eqnarray} the last inequality is due to  Lemma \ref{lem di}.
From the above inequality it follows that \begin{eqnarray}
&&(iii)\;\sup\limits_{z\in \mathbb{D}}\left|\mathcal{T}_{\alpha+m-1}^{\varphi_1}(vu_1)(z)
-\mathcal{T}_{\alpha+m-1}^{\varphi_2}(vu_2)(z)\right|\nonumber\\&&\preceq \sup\limits_{z\in \mathbb{D}}\left( \|(D_{\varphi_1,u_1}^m-D_{\varphi_2,u_2}^m) f_{\varphi_1(z)}\|_v + \left|\mathcal{T}_{\alpha+m-1}^{\varphi_2}
(vu_2)(z)\right|\rho(z)\right)\nonumber\\&&\preceq \sup\limits_{a\in \mathbb{D}}\left( \|(D_{\varphi_1,u_1}^m-D_{\varphi_2,u_2}^m) f_{a}\|_v + \|(D_{\varphi_1,u_1}^m-D_{\varphi_2,u_2}^m) g_{a}\|_v \right). \label{T3}\end{eqnarray}
\eqref{T1}, \eqref{T2} together with \eqref{T3} imply the statement is true. This completes the proof.
\end{proof}
\begin{lemma}\label{lemma FGN}Let $m\in \mathbb{N}_0,$ $0<\alpha<\infty$ and $v$ be a weight. Suppose that $u_1, u_2\in H(\mathbb{D}),$ $\varphi_1, \varphi_2 \in S(\mathbb{D})$, then the following statements hold,
\begin{eqnarray*}&&(i)\;\;\sup\limits_{z\in \mathbb{D}} \|(D_{\varphi_1,u_1}^m-D_{\varphi_2,u_2}^m) f_{a}\|_v \preceq \sup\limits_{n\in \mathbb{N}_0}  n^{\alpha+m-1}\|u_1 \varphi_1^{n} -u_2 \varphi_2^n \|_v;\\&&
(ii)\;\;\sup\limits_{z\in \mathbb{D}} \|(D_{\varphi_1,u_1}^m-D_{\varphi_2,u_2}^m) g_{a}\|_v \preceq \sup\limits_{n\in \mathbb{N}_0}  n^{\alpha+m-1}\|u_1 \varphi_1^{n} -u_2 \varphi_2^n \|_v.\end{eqnarray*} That is to say, $ \sup\limits_{z\in \mathbb{D}} (\|(D_{\varphi_1,u_1}^m-D_{\varphi_2,u_2}^m) f_{a}\|_v +\|(D_{\varphi_1,u_1}^m-D_{\varphi_2,u_2}^m) g_{a}\|_v ) \preceq \sup\limits_{n\in \mathbb{N}_0}  n^{\alpha+m-1}\|u_1 \varphi_1^{n} -u_2 \varphi_2^n \|_v.$
\end{lemma}
\begin{proof} Recall that $$\frac{1}{(1-\bar{a}t_1)^{2\alpha+m-1}}=\sum_{k=0}^\infty \frac{\Gamma(k+2\alpha+m-1)}{\Gamma(2\alpha+m-1) k!}(\bar{a}t_1)^k.$$ Integrating the above display we express $f_a$ and $g_a$ into Maclaurin expansion respectively, as following
\begin{eqnarray*} f_a(z)&=&(1-|a|^2)^\alpha  \int_0^z\int_0^{t_m}\cdots \int_0^{t_2} \sum_{k=0}^\infty \frac{\Gamma(k+2\alpha+m-1)}{\Gamma(2\alpha+m-1) k!}(\bar{a}t_1)^kdt_1dt_2\cdots dt_m\nonumber\\&=& (1-|a|^2)^\alpha\sum_{k=0}^\infty \frac{\Gamma(k+2\alpha+m-1)}{\Gamma(2\alpha+m-1) (k+m)!}(\bar{a} )^k z^{k+m},\;z\in \mathbb{D}. \end{eqnarray*} On the other hand, \begin{eqnarray*}&&\frac{(1-|a|^2)^\alpha}{(1-\bar{a}t_1)^{2\alpha+m-1}}\cdot \frac{a-t_1}{1-\bar{a}t_1} \\&&=(1-|a|^2)^\alpha\left(\sum_{k=0}^\infty \frac{\Gamma(k+2\alpha+m-1)}{\Gamma(2\alpha+m-1)k!} \bar{a}^kt_1^k\right)\left( \frac{a(1-\bar{a}t_1)+|a|^2 t_1-t_1}{1-\bar{a}t_1}\right)\\&&=(1-|a|^2)^\alpha\left(\sum_{k=0}^\infty \frac{\Gamma(k+2\alpha+m-1)}{\Gamma(2\alpha+m-1)k!} \bar{a}^kt_1^k\right)\left( a-(1-|a|^2)\frac{t_1}{1-\bar{a}t_1}\right)\nonumber\\&& =(1-|a|^2)^\alpha\left(\sum_{k=0}^\infty \frac{\Gamma(k+2\alpha+m-1)}{\Gamma(2\alpha+m-1)k!} \bar{a}^kt_1^k\right)\left( a-(1-|a|^2)\sum_{k=0}^\infty \bar{a}^k t_1^{k+1}\right)\nonumber\\&&=a(1-|a|^2)^\alpha\left(\sum_{k=0}^\infty \frac{\Gamma(k+2\alpha+m-1)}{\Gamma(2\alpha+m-1)k!} \bar{a}^kt_1^k\right)\nonumber\\&&-(1-|a|^2)^{\alpha+1} \left(\sum_{k=0}^\infty \frac{\Gamma(k+2\alpha+m-1)}{\Gamma(2\alpha+m-1)k!} \bar{a}^kt_1^k\right)\left( \sum_{k=0}^\infty \bar{a}^k t_1^{k+1}\right)\nonumber\\&&= a(1-|a|^2)^\alpha\left(\sum_{k=0}^\infty \frac{\Gamma(k+2\alpha+m-1)}{\Gamma(2\alpha+m-1)k!} \bar{a}^kt_1^k\right)\nonumber\\&&- (1-|a|^2)^{\alpha+1}\sum_{k=1}^\infty\left( \sum_{l=0}^{k-1} \frac{\Gamma(l+2\alpha+m-1)}{\Gamma(2\alpha+m-1)l!}\right)\bar{a}^{k-1}t_1^k.\end{eqnarray*}
\begin{eqnarray} &&g_a(z) =\int_0^z\int_0^{t_m}\cdots \int_0^{t_2} \frac{(1-|a|^2)^\alpha}{(1-\bar{a}t_1)^{2\alpha+m-1}}\cdot \frac{a-t}{1-\bar{a}t_1} dt_1\cdots dt_m\nonumber\\&&=af_a(z)-(1-|a|^2)^{\alpha+1} \int_0^z\cdots \int_0^{t_2}  \sum_{k=1}^\infty\left( \sum_{l=0}^{k-1} \frac{\Gamma(l+2\alpha+m-1)}{\Gamma(2\alpha+m-1)l!}\bar{a}^{k-1}t_1^k\right)
dt_1\cdots dt_m\nonumber\\&&=af_a(z)- (1-|a|^2)^{\alpha+1} \sum_{k=1}^\infty \frac{k!}{(k+m)!}\left(\sum_{l=0}^{k-1} \frac{\Gamma(l+2\alpha+m-1)}{\Gamma(2\alpha+m-1)l!}\bar{a}^{k-1}z^{k+m} \right).\label{ga}\end{eqnarray}
On  the other hand, we deduce that
\begin{eqnarray}&&\|(D_{\varphi_1,u_1}^m-D_{\varphi_2,u_2}^m)f_a\|_{v}\nonumber\\&\leq & (1-|a|^2)^\alpha\sum_{k=0}^\infty \frac{\Gamma(k+2\alpha+m-1)}{\Gamma(2\alpha+m-1) (k+m)!}|\bar{a} |^k \frac{(k+m)!}{k!} \|u_1 \varphi_1^{k} -u_2 \varphi_2^k \|_v\nonumber\\&=&(1-|a|^2)^\alpha\sum_{k=0}^\infty \frac{\Gamma(k+2\alpha+m-1)}{\Gamma(2\alpha+m-1) k!}|\bar{a} |^k   \|u_1 \varphi_1^{k} -u_2 \varphi_2^k \|_v\label{F} \\&\leq & (1-|a|^2)^\alpha\sum_{k=0}^\infty \frac{\Gamma(k+2\alpha+m-1)}{\Gamma(2\alpha+m-1) k!}|\bar{a} |^k k^{-\alpha-m+1}\nonumber\\&&\cdot\sup\limits_{n\in \mathbb{N}_0}  n^{\alpha+m-1}\|u_1 \varphi_1^{n} -u_2 \varphi_2^n \|_v. \label{F0}\end{eqnarray}
By Stirling's formula, it follows that
\begin{eqnarray*}\left\{
                   \begin{array}{ll}
                     \frac{\Gamma(k+\alpha)}{k!\Gamma(\alpha)} \approx k^{\alpha-1},& \mbox{as}\;k\rightarrow \infty;                     \vspace{5mm} \\ \frac{\Gamma(k+2\alpha+m-1)}{\Gamma(2\alpha+m-1) k!} k^{-\alpha-m+1}\approx k^{\alpha-1}, & \mbox{as}\;k\rightarrow \infty.
                   \end{array}
                 \right.
 \end{eqnarray*} Therefore it yields that
\begin{eqnarray}&&\frac{1}{(1-|a|)^\alpha}=\sum_{k=0}^\infty \frac{\Gamma(k+\alpha)}{k!\Gamma(\alpha)} |a|^k \nonumber\\&\approx& \sum_{k=0}^\infty k^{\alpha-1} |a|^k\nonumber\\&\approx& \sum_{k=0}^\infty \frac{\Gamma(k+2\alpha+m-1)}{\Gamma(2\alpha+m-1) k!} k^{-\alpha-m+1}|\bar{a}|^k. \label{exp} \end{eqnarray}
Putting \eqref{exp} into \eqref{F0}, we deduce that
\begin{eqnarray} \|(D_{\varphi_1,u_1}^m-D_{\varphi_2,u_2}^m)f_a\|_{v}\preceq \sup\limits_{n\in \mathbb{N}_0}  n^{\alpha+m-1}\|u_1 \varphi_1^{n} -u_2 \varphi_2^n \|_v\end{eqnarray}
Using \eqref{ga} it turns out that
\begin{eqnarray*}&&\|(D_{\varphi_1,u_1}^m-D_{\varphi_2,u_2}^m)g_a\|_{v}\preceq \|(D_{\varphi_1,u_1}^m-D_{\varphi_2,u_2}^m)f_a\|_{v} \nonumber\\&&+  (1-|a|^2)^{\alpha+1} \sum_{k=1}^\infty \frac{k!}{(k+m)!}\left(\sum_{l=0}^{k-1} \frac{\Gamma(l+2\alpha+m-1)}{\Gamma(2\alpha+m-1)l!}\right)|\bar{a}|^{k-1}\nonumber\\&&\cdot\frac{(k+m)!}{k!}
\|u_1\varphi_1^{k} -u_2 \varphi_2^k \|_v\nonumber\\&&=|(D_{\varphi_1,u_1}^m-D_{\varphi_2,u_2}^m)f_a\|_{v} \nonumber\\&&+  (1-|a|^2)^{\alpha+1} \sum_{k=1}^\infty \left(\sum_{l=0}^{k-1} \frac{\Gamma(l+2\alpha+m-1)}{\Gamma(2\alpha+m-1)l!}\right)|\bar{a}|^{k-1}
\|u_1 \varphi_1^{k} -u_2 \varphi_2^k \|_v.
 \end{eqnarray*}
 Furthermore by Stirling's formula again, we obtain
 \begin{eqnarray*} &&\sum_{l=0}^{k-1} \frac{\Gamma(l+2\alpha+m-1)}{\Gamma(2\alpha+m-1)l!}\\&\approx& \sum_{l=0}^{k-1} l^{2\alpha+m-2}\\&\approx & k^{2\alpha+m-1},\;\mbox{as}\;k\rightarrow \infty. \end{eqnarray*} For simplicity, we  denote $$a_k= \sum_{l=0}^{k-1} l^{2\alpha+m-2},$$  and then the last equivalence is due to the following fact,   \begin{eqnarray*} && k^{2\alpha+m-1}-(k-1)^{2\alpha+m-1}\\&=&(k-1+1)^{2\alpha+m-1}-(k-1)^{2\alpha+m-1}\nonumber\\&=&
 (k-1)^{2\alpha+m-1}+(2\alpha+m-1)(k-1)^{2\alpha+m-2}+\cdots+1-(k-1)^{2\alpha+m-1}.\end{eqnarray*} We deduce from Stole formula   that
 \begin{eqnarray*} &&\lim\limits_{k\rightarrow \infty} \frac{a_k}{k^{2\alpha+m-1} }\\&=&\lim\limits_{k\rightarrow \infty}\frac{a_k-a_{k-1}}{k^{2\alpha+m-1}-(k-1)^{2\alpha+m-1}}\nonumber\\&=&   \lim\limits_{k\rightarrow \infty}\frac{(k-1)^{2\alpha+m-2} }{k^{2\alpha+m-1}-(k-1)^{2\alpha+m-1}}\nonumber\\&=&\lim\limits_{k\rightarrow \infty}\frac{(k-1)^{2\alpha+m-2} }{(2\alpha+m-1)(k-1)^{2\alpha+m-2}+\cdots+1 }\nonumber\\&=&\frac{1}{2\alpha+m-1}. \end{eqnarray*} Therefore, it yields  that \begin{eqnarray}&&\|(D_{\varphi_1,u_1}^m-D_{\varphi_2,u_2}^m)g_a\|_{v}\preceq \|(D_{\varphi_1,u_1}^m-D_{\varphi_2,u_2}^m)f_a\|_{v}\nonumber\\&&+
 (1-|a|^2)^{\alpha+1} \sum_{k=1}^\infty k^{2\alpha+m-1}|\bar{a}|^{k-1}  \|u_1\varphi_1^{k}-u_2\varphi_2^k\|_v \label{Dga}\\&& \leq \|(D_{\varphi_1,u_1}^m-D_{\varphi_2,u_2}^m)f_a\|_{v}\nonumber\\&&+
 (1-|a|^2)^{\alpha+1} \sum_{k=1}^\infty k^{2\alpha+m-1}|\bar{a}|^{k-1} k^{-\alpha-m+1} \sup\limits_{n\in \mathbb{N}_0} n^{\alpha+m-1}\|u_1 \varphi_1^{n} -u_2 \varphi_2^n \|_v
 \nonumber\\&& \preceq \|(D_{\varphi_1,u_1}^m-D_{\varphi_2,u_2}^m)f_a\|_{v} \nonumber\\&& + (1-|a|^2)^{\alpha+1} \sum_{k=1}^\infty k^{\alpha}|\bar{a}|^{k-1}   \sup\limits_{n\in \mathbb{N}_0} n^{\alpha+m-1} \|u_1 \varphi_1^{n} -u_2 \varphi_2^n \|_v\nonumber\\&&\preceq \|(D_{\varphi_1,u_1}^m-D_{\varphi_2,u_2})f_a\|_{v}\nonumber\\&&+ (1-|a|^2)^{\alpha+1}
 \frac{1}{(1-|a|)^{\alpha+1}}\sup\limits_{n\in \mathbb{N}_0} n^{\alpha+m-1} \|u_1 \varphi_1^{n} -u_2 \varphi_2^n \|_v \nonumber\\&&\preceq \sup\limits_{n\in \mathbb{N}_0} n^{\alpha+m-1} \|u_1 \varphi_1^{n} -u_2 \varphi_2^n \|_v. \nonumber\end{eqnarray} This completes the proof.
  \end{proof}

In this section , our main result is exhibited below:
\begin{theorem} Let $m\in \mathbb{N}_0,$ $0<\alpha<\infty$ and $v$ be a weight. Suppose $u_1, u_2\in H(\mathbb{D}),$ $\varphi_1, \varphi_2 \in S(\mathbb{D})$. Then the following statements are equivalent,\vspace{2mm}

$(i)$ $D_{\varphi_1,u_1}^{m}-D_{\varphi_2,u_2}^{m}: \mathcal{B}^\alpha\rightarrow H_v^\infty$ is bounded;

$(ii)$ \begin{eqnarray*}&&\sup\limits_{z\in \mathbb{D}}|\mathcal{T}_{\alpha+m-1}^{\varphi_1}(vu_1)(z)|\rho(z)+\sup\limits_{z\in \mathbb{D}}\left|\mathcal{T}_{\alpha+m-1}^{\varphi_1}(vu_1)(z) - \mathcal{T}_{\alpha+m-1}^{\varphi_2}(vu_2)(z) \right| <\infty, \quad\quad\quad\quad\label{}\\&&  \sup\limits_{z\in \mathbb{D}}|\mathcal{T}_{\alpha+m-1}^{\varphi_2}(vu_2)(z)|\rho(z)+\sup\limits_{z\in \mathbb{D}}\left|\mathcal{T}_{\alpha+m-1}^{\varphi_1}(vu_1)(z) - \mathcal{T}_{\alpha+m-1}^{\varphi_2}(vu_2)(z) \right| <\infty;
\end{eqnarray*}

$(iii)$ \begin{eqnarray*} \sup\limits_{a\in \mathbb{D}}\|(D_{\varphi_1,u_1}^m-D_{\varphi_2,u_2}^m) f_a\|_v+\sup\limits_{a\in \mathbb{D}} \|(D_{\varphi_1,u_1}^m-D_{\varphi_2,u_2}^m) g_a\|_v <\infty;  \end{eqnarray*}

$(iv)$ \begin{eqnarray*}\sup\limits_{n\in \mathbb{N}_0}  n^{\alpha+m-1}\|u_1 \varphi_1 ^{n}-u_2 \varphi_2^n \|_v<\infty. \end{eqnarray*}
\end{theorem}
\begin{proof} The implications $(iv)\Rightarrow (iii)\Rightarrow (ii)$ yield  from Lemma \ref{lemma FGN} and Lemma \ref{lemma TFG}, respectively. We only need to prove $(i)\Rightarrow (iv)$ and $(ii)\Rightarrow (i)$ below.

$(i)\Rightarrow (iv)$. Suppose that $D_{\varphi_1,u_1}^{m}-D_{\varphi_2,u_2}^{m}: \mathcal{B}^\alpha\rightarrow H_v^\infty$ is bounded; that is, $\|D_{\varphi_1,u_1}^{m}-D_{\varphi_2,u_2}^{m}\|_{\mathcal{B}^\alpha\rightarrow H_v^\infty}<\infty.$  Considering the function $z^n,$ we have known that $\|z^n\|_{\mathcal{B}^\alpha}\approx n^{1-\alpha}$ from \cite[Section 2(6)]{LZ1}. And  then we obtain that
\begin{eqnarray}&&\|D_{\varphi_1,u_1}^{m}-D_{\varphi_2,u_2}^{m}\|_{\mathcal{B}^\alpha\rightarrow H_v^\infty}  \nonumber\\&&\succeq \left\|(D_{\varphi_1,u_1}^{m}-D_{\varphi_2,u_2}^{m})\frac{z^n}{\|z^n\|_{\mathcal{B}^\alpha}}
\right\|_v
\nonumber \\&&\succeq \frac{n!}{(n-m)!n^{1-\alpha}}\|u_1 \varphi_1^{n-m}-u_2\varphi_2^{n-m}\|_v\nonumber\\&& \approx (n-m)^{\alpha+m-1}\|u_1 \varphi_1^{n-m}-u_2\varphi_2^{n-m}\|_v\nonumber\\&& \approx n^{\alpha+m-1}\|u_1 \varphi_1^{n-m}-u_2\varphi_2^{n-m}\|_v. \label{zn}\end{eqnarray}
\eqref{zn} implies  \begin{eqnarray*}\sup\limits_{n\in \mathbb{N}_0}  n^{\alpha+m-1}\|u_1 \varphi_1 ^{n}-u_2 \varphi_2^n \|_v\preceq\|D_{\varphi_1,u_1}^{m}-D_{\varphi_2,u_2}^{m}\|_{\mathcal{B}^\alpha\rightarrow H_v^\infty}  <\infty, \end{eqnarray*}then the implication  $(i)\Rightarrow (iv)$ follows.

$(ii)\Rightarrow (i).$ For any $f\in \mathcal{B}^\alpha,$ we employ  Lemma \ref{lem di} to show  that
\begin{eqnarray}&&\|(D_{\varphi_1,u_1}^m-D_{\varphi_2,u_2}^m) f\|_{H_v^\infty}=\sup\limits_{z\in \mathbb{D}} v(z)|u_1(z)f^{(m)}(\varphi_1(z))-u_2(z)f^{(m)}(z)|\nonumber\\&&\preceq\sup\limits_{z\in \mathbb{D}}\left|\mathcal{T}_{\alpha+m-1}^{\varphi_1}(vu_1)(z)\right|\left|(1-|\varphi_1|^2)^{\alpha+m-1}
f^{(m)}(\varphi_1(z))-(1-|\varphi_2|^2)^{\alpha+m-1} f^{(m)}(\varphi_2(z)) \right|\nonumber\\&&+\sup\limits_{z\in \mathbb{D}} (1-|\varphi_2|^2)^{\alpha+m-1}|f^{(m)}(\varphi_2(z))|\left|\mathcal{T}_{\alpha+m-1}^{\varphi_1}(vu_1)(z)-
\mathcal{T}_{\alpha+m-1}^{\varphi_2}(vu_2)(z)\right|\nonumber\\&&\preceq \sup\limits_{z\in \mathbb{D}} |\mathcal{T}_{\alpha+m-1}^{\varphi_1}(vu_1)(z)|\rho(z) +\sup\limits_{z\in \mathbb{D}}\left|\mathcal{T}_{\alpha+m-1}^{\varphi_1}(vu_1)(z)-
\mathcal{T}_{\alpha+m-1}^{\varphi_2}(vu_2)(z)\right|<\infty.\label{TRHO1}  \end{eqnarray}
Analogously to \eqref{TRHO1}, we can also obtain that
\begin{eqnarray}&&\|(D_{\varphi_1,u_1}^m-D_{\varphi_2,u_2}^m) f\|_{H_v^\infty}\nonumber\\&& \preceq \sup\limits_{z\in \mathbb{D}} |\mathcal{T}_{\alpha+m-1}^{\varphi_2}(vu_2)(z)|\rho(z) +\sup\limits_{z\in \mathbb{D}}\left|\mathcal{T}_{\alpha+m-1}^{\varphi_1}(vu_1)(z)-
\mathcal{T}_{\alpha+m-1}^{\varphi_2}(vu_2)(z)\right|<\infty.\quad\quad\quad\label{TRHO2}  \end{eqnarray} The above two inequalities imply that each one of conditions $(ii)$ can ensure the boundedness of  $D_{\varphi_1,u_1}^{m}-D_{\varphi_2,u_2}^{m}: \mathcal{B}^\alpha\rightarrow H_v^\infty$. This finishes the proof. \end{proof}

\section{ The compactness of $ D_{\varphi_1,u_1}^m-D_{\varphi_2,u_2}^m:  \mathcal{B}^\alpha\rightarrow H_v^\infty$}
 In this section, we  give some lemmas to provide the results for the compactness of $ D_{\varphi_1,u_1}^m-D_{\varphi_2,u_2}^m:  \mathcal{B}^\alpha\rightarrow H_v^\infty$. Firstly, we give some parallel results from Lemma \ref{lemma TFG}  as follows.
 \begin{lemma} \label{lem TGN}Let $m\in \mathbb{N}_0,$ $0<\alpha<\infty$ and $v$ be a weight. Suppose $u_1, u_2\in H(\mathbb{D}),$ $\varphi_1, \varphi_2 \in S(\mathbb{D})$. Then the following inequalities hold,

\begin{eqnarray*}&&(i)\;\;\lim\limits_{r\rightarrow 1}\sup\limits_{|\varphi_1(z)|>r}\left|\mathcal{T}_{\alpha+m-1}^{\varphi_1}(vu_1)(z)\right| \rho(z) \nonumber\\&& \preceq \limsup\limits_{|a|\rightarrow 1} \|(D_{\varphi_1,u_1}^m-D_{\varphi_2,u_2}^m) f_a\|_v +\limsup\limits_{|a|\rightarrow 1} \|(D_{\varphi_1,u_1}^m-D_{\varphi_2,u_2}^m) g_a\|_v. \\&& (ii)\;\; \lim\limits_{r\rightarrow 1}\sup\limits_{|\varphi_2(z)|>r}\left|\mathcal{T}_{\alpha+m-1}^{\varphi_2}(vu_2)(z)\right| \rho(z) \nonumber \\&&\preceq \limsup\limits_{|a|\rightarrow 1}\|(D_{\varphi_1,u_1}^m-D_{\varphi_2,u_2}^m) f_a\|_v +\limsup\limits_{|a|\rightarrow 1}\|(D_{\varphi_1,u_1}^m-D_{\varphi_2,u_2}^m) g_a\|_v.\\&& (iii)\;\;\lim\limits_{r\rightarrow 1}\sup\limits_{\min\{|\varphi_1(z)|,|\varphi_2(z)|\}>r}\left|\mathcal{T}_{\alpha+m-1}^{\varphi_1}(vu_1)(z)-\mathcal{T}_{\alpha+m-1}^{\varphi_2}
(vu_2)(z)\right|\nonumber\\&& \preceq \limsup\limits_{|a|\rightarrow 1}\|(D_{\varphi_1,u_1}^m-D_{\varphi_2,u_2}^m) f_a\|_v +\limsup\limits_{|a|\rightarrow 1}\|(D_{\varphi_1,u_1}^m-D_{\varphi_2,u_2}^m) g_a\|_v.  \end{eqnarray*}\end{lemma}

 \begin{proof} This  results can be deduced from \eqref{T1}--\eqref{T3} in Lemma \ref{lemma TFG}.  \end{proof}

\begin{lemma}\label{lemma limFGN} Let $m\in \mathbb{N}_0,$ $0<\alpha<\infty$ and $v$ be a weight. Suppose $u_1, u_2\in H(\mathbb{D}),$ $\varphi_1, \varphi_2 \in S(\mathbb{D})$. Suppose that $D_{\varphi_1,u_1}^m-D_{\varphi_2,u_2}^m: \mathcal{B}^\alpha\rightarrow H_v^\infty$ is bounded, then the following statements hold,
\begin{eqnarray}&&(i)\;\;\limsup\limits_{|a|\rightarrow 1} \|(D_{\varphi_1,u_1}^m-D_{\varphi_2,u_2}^m) f_{a}\|_v \preceq \limsup\limits_{n\rightarrow \infty}  n^{\alpha+m-1}\|u_1 \varphi_1^{n} -u_2 \varphi_2^n \|_v.\quad\quad\quad\label{DN1}\\&&
(ii)\;\;\limsup\limits_{|a|\rightarrow 1}\|(D_{\varphi_1,u_1}^m-D_{\varphi_2,u_2}^m) g_{a}\|_v \preceq \limsup\limits_{n\rightarrow \infty}  n^{\alpha+m-1}\|u_1 \varphi_1^{n} -u_2 \varphi_2^n \|_v.\label{DN2}\quad\quad\quad\end{eqnarray}
\end{lemma}
\begin{proof}For any $a\in \mathbb{D}$ and each positive integer $N,$ employing \eqref{F}   we obtain
\begin{eqnarray*}&&\|(D_{\varphi_1,u_1}^m-D_{\varphi_2,u_2}^m) f_{a}\|_v\nonumber\\&&\leq (1-|a|^2)^\alpha\sum_{k=0}^\infty \frac{\Gamma(k+2\alpha+m-1)}{\Gamma(2\alpha+m-1) k!}|\bar{a} |^k   \|u_1\varphi_1^{k}-u_2\varphi_2^k\|_v \nonumber\\&&\leq (1-|a|^2)^\alpha\sum_{k=0}^N \frac{\Gamma(k+2\alpha+m-1)}{\Gamma(2\alpha+m-1) k!}|\bar{a} |^k   \|u_1\varphi_1^{k}-u_2\varphi_2^k\|_v\nonumber\\&&+ (1-|a|^2)^\alpha\sum_{k=N+1}^\infty \frac{\Gamma(k+2\alpha+m-1)}{\Gamma(2\alpha+m-1) k!}|\bar{a} |^k   \|u_1\varphi_1^{k}-u_2\varphi_2^k\|_v.
\end{eqnarray*} We denote \begin{eqnarray} &&J_1:= (1-|a|^2)^\alpha\sum_{k=0}^N \frac{\Gamma(k+2\alpha+m-1)}{\Gamma(2\alpha+m-1) k!}|\bar{a} |^k   \|u_1\varphi_1^{k}-u_2\varphi_2^k\|_v;\nonumber\\&& J_2:= (1-|a|^2)^\alpha\sum_{k=N+1}^\infty \frac{\Gamma(k+2\alpha+m-1)}{\Gamma(2\alpha+m-1) k!}|\bar{a} |^k   \|u_1\varphi_1^{k}-u_2\varphi_2^k\|_v.\nonumber\end{eqnarray}
For $k\in\{0,\cdots, N\}$, we can choose $z^{k+m}\in \mathcal{B}^\alpha.$ Using the boundedness of $D_{\varphi_1,u_1}^m-D_{\varphi_2,u_2}^m: \mathcal{B}^\alpha\rightarrow H_v^\infty$, it turns out that $\|u_1\varphi_1^{k}-u_2\varphi_2^k\|_v <\infty.$ Hence $\limsup\limits_{|a|\rightarrow 1} J_1=0.$ On the other hand, it follows from \eqref{exp} that \begin{eqnarray*} J_2&=& (1-|a|^2)^\alpha\sum_{k=N+1}^\infty \frac{\Gamma(k+2\alpha+m-1)}{\Gamma(2\alpha+m-1) k!}|\bar{a} |^k   \|u_1\varphi_1^{k}-u_2\varphi_2^k\|_v\nonumber\\&\leq&(1-|a|^2)^\alpha\sum_{k=N+1}^\infty \frac{\Gamma(k+2\alpha+m-1)}{\Gamma(2\alpha+m-1) k!}|\bar{a} |^k k^{-\alpha-m+1}   \sup\limits_{n\geq N+1 }n^{\alpha+m-1}\|u_1\varphi_1^{n}-u_2\varphi_2^n\|_v\nonumber\\&\preceq& \sup\limits_{n\geq N+1 }n^{\alpha+m-1}\|u_1\varphi_1^{n}-u_2\varphi_2^n\|_v. \end{eqnarray*} Furthermore, letting $|a|\rightarrow 1$, it leads to
\begin{eqnarray*}\limsup\limits_{|a|\rightarrow 1}  J_2  \preceq  \sup\limits_{n\geq N+1 }n^{\alpha+m-1}\|u_1\varphi_1^{n}-u_2\varphi_2^n\|_v. \end{eqnarray*} The above inequalities imply that  \eqref{DN1} is true. Similarly, by \eqref{Dga}
\begin{eqnarray*} &&\|(D_{\varphi_1,u_1}^m-D_{\varphi_2,u_2}^m)g_a\|_{v}\preceq \|(D_{\varphi_1,u_1}^m-D_{\varphi_2,u_2}^m)f_a\|_{v}\nonumber\\&&+
 (1-|a|^2)^{\alpha+1} \sum_{k=1}^\infty k^{2\alpha+m-1}|\bar{a}|^{k-1}  \|u_1 \varphi_1^{k} -u_2 \varphi_2^k\|_v\nonumber\\&& \leq \|(D_{\varphi_1,u_1}^m-D_{\varphi_2,u_2}^m)f_a\|_{v}\nonumber\\&&+
 (1-|a|^2)^{\alpha+1} \sum_{k=1}^N k^{2\alpha+m-1}|\bar{a}|^{k-1}  \|u_1\varphi_1^{k}-u_2\varphi_2^k\|_v\nonumber\\&& +
 (1-|a|^2)^{\alpha+1} \sum_{k=N+1}^\infty k^{2\alpha+m-1}|\bar{a}|^{k-1}  \|u_1\varphi_1^{k}-u_2\varphi_2^k\|_v \nonumber\\&& \leq \|(D_{\varphi_1,u_1}^m-D_{\varphi_2,u_2}^m)f_a\|_{v}\nonumber\\&&+
 (1-|a|^2)^{\alpha+1} \sum_{k=1}^N k^{2\alpha+m-1}|\bar{a}|^{k-1}  \|u_1\varphi_1^{k}-u_2\varphi_2^k\|_v\nonumber\\&& +
 (1-|a|^2)^{\alpha+1} \sum_{k=N+1}^\infty k^{2\alpha+m-1}|\bar{a}|^{k-1} k^{-\alpha-m+1} \sup\limits_{n\geq N+1}n^{\alpha+m-1} \|u_1\varphi_1^n-u_2\varphi_2^n\|_v\nonumber\\&& \preceq \|(D_{\varphi_1,u_1}^m-D_{\varphi_2,u_2}^m)f_a\|_{v}\nonumber\\&&+
 (1-|a|^2)^{\alpha+1} \sum_{k=1}^N k^{2\alpha+m-1}|\bar{a}|^{k-1}  \|u_1\varphi_1^{k}-u_2\varphi_2^k\|_v\nonumber\\&& +
  \sup\limits_{n\geq N+1}n^{\alpha+m-1} \|u_1\varphi_1^n-u_2\varphi_2^n\|_v.\end{eqnarray*}
Letting $|a|\rightarrow 1$ in the above display, we get that
\begin{eqnarray*} &&\limsup\limits_{|a|\rightarrow 1}\|(D_{\varphi_1,u_1}^m-D_{\varphi_2,u_2}^m)g_a\|_{v}\nonumber\\&&
\preceq \limsup\limits_{|a|\rightarrow 1} \|(D_{\varphi_1,u_1}^m-D_{\varphi_2,u_2}^m)f_a\|_{v}
+\sup\limits_{n\geq N+1}n^{\alpha+m-1} \|u_1\varphi_1^n-u_2\varphi_2^n\|_v.\end{eqnarray*} The above inequality together with \eqref{DN1} verify \eqref{DN2}. This ends the proof.
\end{proof}
Here we will use the similar  methods in \cite[Section 4]{LZ1}, we let $K_r f(z)=f(rz)$ for $r\in (0,1).$ And then $K_r$ is a compact operator on the Bloch type space $\mathcal{B}^\alpha$ or the little Bloch type space $\mathcal{B}^\alpha_0$ for $\alpha>0$ with $\|K_r\|\leq 1.$ Here we combine the cases for $0<\alpha<1, $ $\alpha=1$ and $\alpha>1$ for Bloch type space.
\begin{lemma}\cite[Lemma 4.1-4.3]{Zhao} \label{lemma Ln} Let $0<\alpha<\infty.$ Then there is a sequence $\{r_k\},$ with $0<r_k<1$ tending to $1$, such that the compact operator
$$L_n=\frac{1}{n} \sum_{k=1}^n K_{r_k}$$ on $\mathcal{B}_0^\alpha$ satisfies $\lim\limits_{n\rightarrow \infty} \sup\limits_{\|f\|_{\mathcal{B}^\alpha}\leq 1}\sup\limits_{|z|\leq t}|((I-L_n)f)'(z)|=0$ for any $t\in[0,1).$ Furthermore, this statement  holds as well for the sequence of biadjoints $L_n^{**}$ on
$\mathcal{B}^\alpha$.
 \end{lemma}
The following is our main theorem in this section.
\begin{theorem} Let $m\in \mathbb{N}_0,$ $0<\alpha<\infty$ and $v$ be a weight. Suppose $u_1, u_2\in H(\mathbb{D}),$ $\varphi_1, \varphi_2 \in S(\mathbb{D})$. Suppose that $D_{\varphi_i,u_i}^m : \mathcal{B}^\alpha\rightarrow H_v^\infty$ is bounded for $i=1,2$, then  the following equivalences hold, \begin{eqnarray*}&&\|D_{\varphi_1,u_1}^m-D_{\varphi_2,u_2}^m\|_{e, \mathcal{B}^\alpha\rightarrow H_v^\infty}\nonumber\\&\approx& \lim\limits_{r\rightarrow 1}\sup\limits_{|\varphi_1(z)|>r}\left|\mathcal{T}_{\alpha+m-1}^{\varphi_1}(vu_1)(z)\right| \rho(z) \nonumber\\&&+\lim\limits_{r\rightarrow 1}\sup\limits_{|\varphi_2(z)|>r}\left|\mathcal{T}_{\alpha+m-1}^{\varphi_2}(vu_2)(z)\right| \rho(z) \nonumber\\&&+ \lim\limits_{r\rightarrow 1}\sup\limits_{\min\{|\varphi_1(z)|,|\varphi_2(z)|\}>r}\left|\mathcal{T}_{\alpha+m-1}^{\varphi_1}(vu_1)(z)-\mathcal{T}_{\alpha+m-1}^{\varphi_2}
(vu_2)(z)\right| \nonumber\\&\approx& \limsup\limits_{|a|\rightarrow 1}\|(D_{\varphi_1,u_1}^m-D_{\varphi_2,u_2}^m) f_{a}\|_v+\limsup\limits_{|a|\rightarrow 1}\|(D_{\varphi_1,u_1}^m-D_{\varphi_2,u_2}^m) g_{a}\|_v \nonumber\\&\approx& \limsup\limits_{n\rightarrow \infty}  n^{\alpha+m-1}\|u_1 \varphi_1^{n} -u_2 \varphi_2^n \|_v.  \end{eqnarray*}\end{theorem}
\begin{proof} Firstly, the boundedness of  $D_{\varphi_i,u_i}^m : \mathcal{B}^\alpha\rightarrow H_v^\infty$ implies that $M_i=\sup\limits_{z\in \mathbb{D}}v(z)|u_i(z)|<\infty$ for $i=1,2.$ Lemma \ref{lemma FGN} together with Lemma \ref{lem TGN} ensure that
\begin{eqnarray*}&& \lim\limits_{r\rightarrow 1}\sup\limits_{|\varphi_1(z)|>r}\left|\mathcal{T}_{\alpha+m-1}^{\varphi_1}(vu_1)(z)\right| \rho(z) +\lim\limits_{r\rightarrow 1}\sup\limits_{|\varphi_2(z)|>r}\left|\mathcal{T}_{\alpha+m-1}^{\varphi_2}(vu_2)(z)\right| \rho(z) \nonumber\\&&+\lim\limits_{r\rightarrow 1}\sup\limits_{\min\{|\varphi_1(z)|,|\varphi_2(z)|\}>r}\left|\mathcal{T}_{\alpha+m-1}^{\varphi_1}(vu_1)(z)-\mathcal{T}_{\alpha+m-1}^{\varphi_2}
(vu_2)(z)\right|\nonumber\\&& \preceq \limsup\limits_{|a|\rightarrow 1}\|(D_{\varphi_1,u_1}^m-D_{\varphi_2,u_2}^m) f_a\|_v +\limsup\limits_{|a|\rightarrow 1}\|(D_{\varphi_1,u_1}^m-D_{\varphi_2,u_2}^m) g_a\|_v\nonumber\\&&\preceq  \limsup\limits_{n\rightarrow \infty}  n^{\alpha+m-1}\|u_1 \varphi_1^{n} -u_2 \varphi_2^n \|_v.  \end{eqnarray*}
In the following, we only need to show \begin{eqnarray*}&& \limsup\limits_{n\rightarrow \infty}  n^{\alpha+m-1}\|u_1 \varphi_1^{n} -u_2 \varphi_2^n \|_v\preceq \|D_{\varphi_1,u_1}^m-D_{\varphi_2,u_2}^m\|_{e, \mathcal{B}^\alpha\rightarrow H_v^\infty}\nonumber\\&&\preceq \lim\limits_{r\rightarrow 1}\sup\limits_{|\varphi_1(z)|>r}\left|\mathcal{T}_{\alpha+m-1}^{\varphi_1}(vu_1)(z)\right| \rho(z) +\lim\limits_{r\rightarrow 1}\sup\limits_{|\varphi_2(z)|>r}\left|\mathcal{T}_{\alpha+m-1}^{\varphi_2}(vu_2)(z)\right| \rho(z) \nonumber\\&&+\lim\limits_{r\rightarrow 1}\sup\limits_{\min\{|\varphi_1(z)|,|\varphi_2(z)|\}>r}\left|\mathcal{T}_{\alpha+m-1}^{\varphi_1}(vu_1)(z)-\mathcal{T}_{\alpha+m-1}^{\varphi_2}
(vu_2)(z)\right|.  \end{eqnarray*} The first inequality follows from the fact: choose a sequence $f_n(z)= z^n/\|z^n\|_{\mathcal{B}^\alpha},$ which converges to $0$ in $\mathcal{B}^\alpha$ with $\|f_n\|_{\mathcal{B}^\alpha}=1.$  For any compact operator $K: \mathcal{B}^\alpha\rightarrow H_v^\infty$, it yields that $\lim\limits_{n\rightarrow \infty}\|Kf_n\|_v=0.$ Furthermore, we deduce that
\begin{eqnarray*}&& \|D_{\varphi_1,u_1}^m-D_{\varphi_2,u_2}^m\|_{e, \mathcal{B}^\alpha\rightarrow H_v^\infty} \nonumber\\&&\succeq  \limsup\limits_{n\rightarrow \infty} \inf_{K}\|(D_{\varphi_1,u_1}^m-D_{\varphi_2,u_2}^m-K) f_n\|_v\nonumber\\&& \geq \limsup\limits_{n\rightarrow \infty} \inf_{K}\left(\|(D_{\varphi_1,u_1}^m-D_{\varphi_2,u_2}^m) f_n\|_v-\|Kf_n\|_v\right)\nonumber\\&& \geq \limsup\limits_{n\rightarrow \infty} \|(D_{\varphi_1,u_1}^m-D_{\varphi_2,u_2}^m) f_n\|_v\nonumber\\&&\succeq \limsup\limits_{n\rightarrow \infty}  n^{\alpha+m-1}\|u_1 \varphi_1^{n-m}-u_2\varphi_2^{n-m}\|_v,\end{eqnarray*} the last inequality follows from \eqref{zn}.

Now we turn our attention to the second inequality. Let $\{L_n\}$ be the sequence of operators given in Lemma \ref{lemma Ln}. Since $L_n^{**}$ is compact on $\mathcal{B}^\alpha$ and $D_{\varphi_1,u_1}^m-D_{\varphi_2,u_2}^m$ is bounded from $\mathcal{B}^\alpha$ to $ H_v^\infty,$ the operator $(D_{\varphi_1,u_1}^m-D_{\varphi_2,u_2}^m)L_n^{**} : \mathcal{B}^\alpha\rightarrow H_v^\infty$ is also compact. Therefore, it follows that
\begin{eqnarray*}&& \|D_{\varphi_1,u_1}^m-D_{\varphi_2,u_2}^m\|_{e, \mathcal{B}^\alpha\rightarrow H_v^\infty} \nonumber\\&&\preceq \limsup\limits_{n\rightarrow \infty} \|(D_{\varphi_1,u_1}^m-D_{\varphi_2,u_2}^m)-(D_{\varphi_1,u_1}^m-D_{\varphi_2,u_2}^m) L_n^{**}\|_{ \mathcal{B}^\alpha\rightarrow H_v^\infty}\nonumber\\&& =\limsup\limits_{n\rightarrow \infty} \|(D_{\varphi_1,u_1}^m-D_{\varphi_2,u_2}^m)(I- L_n^{**})\|_{\mathcal{B}^\alpha\rightarrow H_v^\infty}\nonumber\\&& =\limsup\limits_{n\rightarrow \infty} \sup\limits_{\|f\|_{\mathcal{B}^\alpha}\leq 1} \|(D_{\varphi_1,u_1}^m-D_{\varphi_2,u_2}^m)(I- L_n^{**})f\|_v\nonumber\\&&= \limsup\limits_{n\rightarrow \infty} \sup\limits_{\|f\|_{\mathcal{B}^\alpha}\leq 1} \sup\limits_{z\in \mathbb{D}}v(z)\left|u_1(z)[(I- L_n^{**})f]^{(m)}(\varphi_1(z))-u_2(z)[(I- L_n^{**})f]^{(m)}(\varphi_2(z))\right|.
 \end{eqnarray*}
For an arbitrary $r\in (0,1)$, we denote \begin{eqnarray*}&&\mathbb{D}_1=\{z\in \mathbb{D}:\;|\varphi_1(z)|\leq r,\;|\varphi_2(z)|\leq r\},\; \mathbb{D}_2=\{z\in \mathbb{D}:\;|\varphi_1(z)|\leq r, \;|\varphi_2(z)|>r\},\\&& \mathbb{D}_3=\{z\in \mathbb{D}:|\varphi_1(z)|>r,\;|\varphi_2(z)|\leq r\},\;\mathbb{D}_4=\{z\in \mathbb{D}:\;|\varphi_1(z)|>r, \;|\varphi_2(z)|>r\};\\&& I_i:=\sup\limits_{z\in \mathbb{D}_i}v(z)\left|u_1(z)[(I- L_n^{**})f]^{(m)}(\varphi_1(z))-u_2(z)[(I- L_n^{**})f]^{(m)}(\varphi_2(z))\right|,
 \end{eqnarray*} $\;\mbox{for}\;i=1,2,3,4.$ Then  Cauchy's integral formula and Lemma \ref{lemma Ln} imply that
 \begin{eqnarray}&& \limsup\limits_{n\rightarrow \infty} \sup\limits_{\|f\|_{\mathcal{B}^\alpha}\leq 1}I_1\nonumber\\&&\leq
 \limsup\limits_{n\rightarrow \infty} \sup\limits_{\|f\|_{\mathcal{B}^\alpha}\leq 1} \sup_{|\varphi_1(z)|\leq r}(v(z)|u_1(z)|)|[(I-L_n^{**})f]^{(m)}(\varphi_1(z))|\nonumber\\&& + \limsup\limits_{n\rightarrow \infty} \sup\limits_{\|f\|_{\mathcal{B}^\alpha}\leq 1}  \sup_{|\varphi_2(z)|\leq r}(v(z)|u_2(z)|)|[(I-L_n^{**})f]^{(m)}(\varphi_2(z))|\nonumber\\&& \leq M_1
 \limsup\limits_{n\rightarrow \infty} \sup\limits_{\|f\|_{\mathcal{B}^\alpha}\leq 1} \sup_{|\varphi_1(z)|\leq r} |[(I-L_n^{**})f]^{(m)}(\varphi_1(z))|\nonumber\\&& + M_2 \limsup\limits_{n\rightarrow \infty} \sup\limits_{\|f\|_{\mathcal{B}^\alpha}\leq 1}  \sup_{|\varphi_2(z)|\leq r} |[(I-L_n^{**})f]^{(m)}(\varphi_2(z))|
 =0. \label{I1}\end{eqnarray} On the other hand, we formulate that
 \begin{eqnarray} &&v(z)\left|u_1(z)[(I- L_n^{**})f]^{(m)}(\varphi_1(z))-u_2(z)[(I- L_n^{**})f]^{(m)}(\varphi_2(z))\right|\nonumber\\&&\preceq \frac{v(z)|u_1(z)|}{(1-|\varphi_1(z)|^2)^{\alpha+m-1}}|(1-|\varphi_1(z)|^2)^{\alpha+m-1}[(I- L_n^{**})f]^{(m)}(\varphi_1(z))\nonumber\\&&-(1-|\varphi_2(z)|^2)^{\alpha+m-1}[(I- L_n^{**})f]^{(m)}(\varphi_2(z))|\nonumber\\&&+(1-|\varphi_2(z)|^2)^{\alpha+m-1}\left|[(I- L_n^{**})f]^{(m)}(\varphi_2(z))\right| \left|\mathcal{T}_{\alpha+m-1}^{\varphi_1}(vu_1)(z)-\mathcal{T}_{\alpha+m-1}^{\varphi_2}
(vu_2)(z)\right|\nonumber\\&& \preceq  |\mathcal{T}_{\alpha+m-1}^{\varphi_1}(vu_1)(z)|\rho(z)+(1-|\varphi_2(z)|^2)^{\alpha+m-1}\nonumber\\&&\cdot\left|[(I- L_n^{**})f]^{(m)}(\varphi_2(z)) \right| \left|\mathcal{T}_{\alpha+m-1}^{\varphi_1}(vu_1)(z)-\mathcal{T}_{\alpha+m-1}^{\varphi_2}
(vu_2)(z)\right|.\label{F1}\end{eqnarray}
Analogously, we obtain that  \begin{eqnarray} &&v(z)\left|u_1(z)[(I- L_n^{**})f]^{(m)}(\varphi_1(z))-u_2(z)[(I- L_n^{**})f]^{(m)}(\varphi_2(z))\right|\nonumber\\&&\preceq
|\mathcal{T}_{\alpha+m-1}^{\varphi_2}(vu_2)(z)|\rho(z)
+(1-|\varphi_1(z)|^2)^{\alpha+m-1}\nonumber\\&& \cdot\left|[(I- L_n^{**})f]^{(m)}(\varphi_1(z)) \right| \left|\mathcal{T}_{\alpha+m-1}^{\varphi_1}(vu_1)(z)-\mathcal{T}_{\alpha+m-1}^{\varphi_2}
(vu_2)(z)\right|. \label{F2} \end{eqnarray} Now employ Lemma \ref{lemma Ln} and the boundedness of $\left|\mathcal{T}_{\alpha+m-1}^{\varphi_1}(vu_1)(z)-\mathcal{T}_{\alpha+m-1}^{\varphi_2}
(vu_2)(z)\right|$ in \eqref{F2} to show  that
 \begin{eqnarray}&& \limsup\limits_{n\rightarrow \infty} \sup\limits_{\|f\|_{\mathcal{B}^\alpha}\leq 1}I_2\nonumber\\&\leq& \limsup\limits_{n\rightarrow \infty} \sup\limits_{\|f\|_{\mathcal{B}^\alpha}\leq 1} \sup\limits_{|\varphi_2(z)|>r}  |\mathcal{T}_{\alpha+m-1}^{\varphi_2}(vu_2)(z)|\rho(z)
\nonumber\\&&+\limsup\limits_{n\rightarrow \infty} \sup\limits_{\|f\|_{\mathcal{B}^\alpha}\leq 1}\sup\limits_{|\varphi_1(z)|\leq r}(1-|\varphi_1(z)|^2)^{\alpha+m-1} \left|[(I- L_n^{**})f]^{(m)}(\varphi_1(z)) \right| \nonumber\\&& \cdot \left|\mathcal{T}_{\alpha+m-1}^{\varphi_1}(vu_1)(z)-\mathcal{T}_{\alpha+m-1}^{\varphi_2}
(vu_2)(z)\right|\nonumber\\&\preceq& \sup\limits_{|\varphi_2(z)|>r}  |\mathcal{T}_{\alpha+m-1}^{\varphi_2}(vu_2)(z)|\rho(z). \label{I2}\end{eqnarray}
Similarly, employing \eqref{F1} we deduce that
\begin{eqnarray}&& \limsup\limits_{n\rightarrow \infty} \sup\limits_{\|f\|_{\mathcal{B}^\alpha}\leq 1}I_3 \preceq \sup\limits_{|\varphi_1(z)|>r}  |\mathcal{T}_{\alpha+m-1}^{\varphi_1}(vu_1)(z)|\rho(z). \label{I3} \end{eqnarray} Finally, we deduce from \eqref{F1}  that
\begin{eqnarray} && \limsup\limits_{n\rightarrow \infty} \sup\limits_{\|f\|_{\mathcal{B}^\alpha}\leq 1}I_4 \preceq \sup\limits_{|\varphi_1(z)|>r} |\mathcal{T}_{\alpha+m-1}^{\varphi_1}(vu_1)(z)|\rho(z) \nonumber\\&&+\left\|(I- L_n^{**})f\right\|_{\mathcal{B}^\alpha} \sup\limits_{\min\{|\varphi_1(z)|,|\varphi_2(z)|\}>r} \left|\mathcal{T}_{\alpha+m-1}^{\varphi_1}(vu_1)(z)-\mathcal{T}_{\alpha+m-1}^{\varphi_2}
(vu_2)(z)\right|\nonumber\\&&  \preceq \sup\limits_{|\varphi_1(z)|>r} \mathcal{T}_{\alpha+m-1}^{\varphi_1}(vu_1)(z)\rho(z) \nonumber\\&&+  \sup\limits_{\min\{|\varphi_1(z)|,|\varphi_2(z)|\}>r} \left|\mathcal{T}_{\alpha+m-1}^{\varphi_1}(vu_1)(z)-\mathcal{T}_{\alpha+m-1}^{\varphi_2}
(vu_2)(z)\right|.\label{I41}  \end{eqnarray} Consequently,   \eqref{F2}  entails that
\begin{eqnarray} && \limsup\limits_{n\rightarrow \infty} \sup\limits_{\|f\|_{\mathcal{B}^\alpha}\leq 1}I_4 \preceq \sup\limits_{|\varphi_2(z)|>r} |\mathcal{T}_{\alpha+m-1}^{\varphi_2}(vu_2)(z)|\rho(z) \nonumber\\&&+  \sup\limits_{\min\{|\varphi_1(z)|,|\varphi_2(z)|\}>r} \left|\mathcal{T}_{\alpha+m-1}^{\varphi_1}(vu_1)(z)-\mathcal{T}_{\alpha+m-1}^{\varphi_2}
(vu_2)(z)\right|. \label{I42} \end{eqnarray}
Combining \eqref{I1}, \eqref{I2}, \eqref{I3} and \eqref{I41}, \eqref{I42}, we find that
\begin{eqnarray*}   && \|D_{\varphi_1,u_1}^m-D_{\varphi_2,u_2}^m\|_{e, \mathcal{B}^\alpha\rightarrow H_v^\infty}\nonumber\\&&\preceq \lim\limits_{r\rightarrow 1}\sup\limits_{|\varphi_1(z)|>r}\left|\mathcal{T}_{\alpha+m-1}^{\varphi_1}
(vu_1)(z)\right| \rho(z) +\lim\limits_{r\rightarrow 1}\sup\limits_{|\varphi_2(z)|>r}\left|\mathcal{T}_{\alpha+m-1}^{\varphi_2}(vu_2)(z)\right| \rho(z) \nonumber\\&&+\lim\limits_{r\rightarrow 1}\sup\limits_{\min\{|\varphi_1(z)|,|\varphi_2(z)|\}>r}\left|\mathcal{T}_{\alpha+m-1}^{\varphi_1}(vu_1)(z)-\mathcal{T}_{\alpha+m-1}^{\varphi_2}
(vu_2)(z)\right|.  \end{eqnarray*} This ends all the proof for the essential norm estimation.
\end{proof}
At last we give three equivalent characterizations for the compactness of $D_{\varphi_1,u_1}^m- D_{\varphi_2,u_2}^m: \mathcal{B}^\alpha\rightarrow H_v^\infty$.

 \begin{theorem} Let $m\in \mathbb{N}_0,$ $0<\alpha<\infty$ and $v$ be a weight. Suppose $u_1, u_2\in H(\mathbb{D}),$ $\varphi_1, \varphi_2 \in S(\mathbb{D})$. Suppose that $D_{\varphi_i,u_i}^m : \mathcal{B}^\alpha\rightarrow H_v^\infty$ is bounded for $i=1,2$, then  $D_{\varphi_1,u_1}^m- D_{\varphi_2,u_2}^m: \mathcal{B}^\alpha\rightarrow H_v^\infty$ is compact if and only if one of the following statements hold, \begin{eqnarray*}&&(i)\; \lim\limits_{r\rightarrow 1}\sup\limits_{|\varphi_1(z)|>r}\left|\mathcal{T}_{\alpha+m-1}^{\varphi_1}(vu_1)(z)\right| \rho(z) +\lim\limits_{r\rightarrow 1}\sup\limits_{|\varphi_2(z)|>r}\left|\mathcal{T}_{\alpha+m-1}^{\varphi_2}(vu_2)(z)\right| \rho(z) \nonumber\\&&+ \lim\limits_{r\rightarrow 1}\sup\limits_{\min\{|\varphi_1(z)|,|\varphi_2(z)|\}>r}\left|\mathcal{T}_{\alpha+m-1}^{\varphi_1}(vu_1)(z)-\mathcal{T}_{\alpha+m-1}^{\varphi_2}
(vu_2)(z)\right|=0; \nonumber\\&&(ii)\; \limsup\limits_{|a|\rightarrow 1}\|(D_{\varphi_1,u_1}^m-D_{\varphi_2,u_2}^m) f_{a}\|_v+\limsup\limits_{|a|\rightarrow 1}\|(D_{\varphi_1,u_1}^m-D_{\varphi_2,u_2}^m) g_{a}\|_v =0;\nonumber\\&& (iii)\;\limsup\limits_{n\rightarrow \infty}  n^{\alpha+m-1}\|u_1 \varphi_1^{n} -u_2 \varphi_2^n \|_v=0.  \end{eqnarray*}\end{theorem}

\section{Some Corollaries}
In this section, we listed some corollaries for the boundedness and compactness for the difference of several classical operators, such as, $C_{\varphi_1}-C_{\varphi_2}$, $u_1C_{\varphi_1}-u_2C_{\varphi_2},$ $C_{\varphi_1}D- C_{\varphi_2}D$ and $u_1 C_{\varphi_1}D-u_2 C_{\varphi_2}D$ from $\mathcal{B}^\alpha$ into $H_v^\infty.$\vspace{1mm}

\textbf{Case I}:\; Let $u_1=u_2=id$, the identity map,  and $\;m=0,$  then $D_{\varphi_1,id}^0-D_{\varphi_2,id}^0=C_{\varphi_1}-C_{\varphi_2}.$  And in this case, we still use the notations $f_a$ and $g_a$ to stand for the following test functions,
\begin{eqnarray*}&&f_a(z)= \frac{(1-|a|^2)^\alpha}{(1-\bar{a}z)^{2\alpha+m-1}},\\&& g_a(z)=  \frac{(1-|a|^2)^\alpha}{(1-\bar{a}z)^{2\alpha+m-1}}\cdot \frac{a-z}{1-\bar{a}z}.\end{eqnarray*}
\begin{corollary} Let  $0<\alpha<\infty$ and $v$ be a weight. Suppose  $\varphi_1, \varphi_2 \in S(\mathbb{D})$. Then the following statements are equivalent,

$(i)$ $C_{\varphi_1} -C_{\varphi_2}: \mathcal{B}^\alpha\rightarrow H_v^\infty$ is bounded;

$(ii)$ \begin{eqnarray}&&\sup\limits_{z\in \mathbb{D}}|\mathcal{T}_{\alpha-1}^{\varphi_1}(v)(z)|\rho(z)+\sup\limits_{z\in \mathbb{D}}\left|\mathcal{T}_{\alpha-1}^{\varphi_1}(v)(z) - \mathcal{T}_{\alpha-1}^{\varphi_2}(v)(z) \right| <\infty, \nonumber\\&&  \sup\limits_{z\in \mathbb{D}}|\mathcal{T}_{\alpha-1}^{\varphi_2}(v)(z)|\rho(z)+\sup\limits_{z\in \mathbb{D}}\left|\mathcal{T}_{\alpha-1}^{\varphi_1}(v)(z) - \mathcal{T}_{\alpha-1}^{\varphi_2}(v)(z) \right| <\infty;\nonumber
\end{eqnarray}

$(iii)$ \begin{eqnarray*} \sup\limits_{a\in \mathbb{D}}\|(C_{\varphi_1}-C_{\varphi_2}) f_a\|_v+\sup\limits_{a\in \mathbb{D}} \|(C_{\varphi_1}-C_{\varphi_2}) g_a\|_v <\infty; \end{eqnarray*}

$(iv)$ \begin{eqnarray*}\sup\limits_{n\in \mathbb{N}_0}  n^{\alpha-1}\| \varphi_1^n- \varphi_2^n\|_v<\infty.  \end{eqnarray*}
\end{corollary}

 \begin{corollary} Let $0<\alpha<\infty$, $v$ be a weight and $\varphi_1, \varphi_2 \in S(\mathbb{D})$. Suppose that $C_{\varphi_i}: \mathcal{B}^\alpha\rightarrow H_v^\infty$ is bounded for $i=1,2$. Then    the following statements  are equivalent, \begin{eqnarray*}&&(i) \mbox{ $C_{\varphi_1}-C_{\varphi_2}: \mathcal{B}^\alpha\rightarrow H_v^\infty$ is compact};
 \\&&(ii)\; \lim\limits_{r\rightarrow 1}\sup\limits_{|\varphi_1(z)|>r}\left|\mathcal{T}_{\alpha-1}^{\varphi_1}(v)(z)\right| \rho(z) +\lim\limits_{r\rightarrow 1}\sup\limits_{|\varphi_2(z)|>r}\left|\mathcal{T}_{\alpha-1}^{\varphi_2}(v)(z)\right| \rho(z) \nonumber\\&&+ \lim\limits_{r\rightarrow 1}\sup\limits_{\min\{|\varphi_1(z)|,|\varphi_2(z)|\}>r}\left|\mathcal{T}_{\alpha-1}^{\varphi_1}(v)(z)
 -\mathcal{T}_{\alpha-1}^{\varphi_2}
(v)(z)\right|=0; \nonumber\\&&(iii)\; \limsup\limits_{|a|\rightarrow 1}\|(C_{\varphi_1}-C_{\varphi_2}) f_{a}\|_v+\limsup\limits_{|a|\rightarrow 1}\|(C_{\varphi_1}-C_{\varphi_2} ) g_{a}\|_v =0;\nonumber\\&& (iv)\;\limsup\limits_{n\rightarrow \infty}  n^{\alpha-1}\| \varphi_1^{n} -\varphi_2^n \|_v=0.  \end{eqnarray*}\end{corollary}

\textbf{Case II}:\; Let $\;m=0,$  then $D_{\varphi_1,u_1}^0-D_{\varphi_2,u_2}^0=u_1C_{\varphi_1}-u_2C_{\varphi_2}.$  And in this case, we still use $f_a$ and $g_a$ as below.
\begin{eqnarray*}&&f_a(z)= \frac{(1-|a|^2)^\alpha}{(1-\bar{a}z)^{2\alpha+m-1}},\\&& g_a(z)=  \frac{(1-|a|^2)^\alpha}{(1-\bar{a}z)^{2\alpha+m-1}}\cdot \frac{a-z}{1-\bar{a}z}.\end{eqnarray*}

\begin{corollary} Let $0<\alpha<\infty$ and $v$ be a weight. Suppose $u_1, u_2\in H(\mathbb{D})$ and $\varphi_1, \varphi_2 \in S(\mathbb{D})$. Then the following statements are equivalent,

$(i)$ $u_1C_{\varphi_1}-u_2C_{\varphi_2}: \mathcal{B}^\alpha\rightarrow H_v^\infty$ is bounded;

$(ii)$ \begin{eqnarray*}&&\sup\limits_{z\in \mathbb{D}}|\mathcal{T}_{\alpha -1}^{\varphi_1}(vu_1)(z)|\rho(z)+\sup\limits_{z\in \mathbb{D}}\left|\mathcal{T}_{\alpha-1}^{\varphi_1}(vu_1)(z) - \mathcal{T}_{\alpha-1}^{\varphi_2}(vu_2)(z) \right| <\infty,  \\&&  \sup\limits_{z\in \mathbb{D}}|\mathcal{T}_{\alpha-1}^{\varphi_2}(vu_2)(z)|\rho(z)+\sup\limits_{z\in \mathbb{D}}\left|\mathcal{T}_{\alpha-1}^{\varphi_1}(vu_1)(z) - \mathcal{T}_{\alpha-1}^{\varphi_2}(vu_2)(z) \right| <\infty;
\end{eqnarray*}

$(iii)$ \begin{eqnarray*} \sup\limits_{a\in \mathbb{D}}\|(u_1C_{\varphi_1}-u_2C_{\varphi_2}) f_a\|_v+\sup\limits_{a\in \mathbb{D}} \|(u_1C_{\varphi_1}-u_2C_{\varphi_2}) g_a\|_v <\infty; \end{eqnarray*}

$(iv)$ \begin{eqnarray*}\sup\limits_{n\in \mathbb{N}_0}  n^{\alpha-1}\|u_1\varphi_1^{n}-u_2\varphi_2^n\|_v<\infty. \end{eqnarray*}
\end{corollary}

\begin{corollary} Let  $0<\alpha<\infty$ and $v$ be a weight. Suppose $u_1, u_2\in H(\mathbb{D})$ and $\varphi_1, \varphi_2 \in S(\mathbb{D})$. Suppose that $u_iC_{\varphi_i}: \mathcal{B}^\alpha\rightarrow H_v^\infty$ is bounded for $i=1,2$, then  the following statements are equivalent,  \begin{eqnarray*}&&(i)\; \mbox{$u_1C_{\varphi_1}-u_2 C_{\varphi_2}: \mathcal{B}^\alpha\rightarrow H_v^\infty$ is compact;}\nonumber\\&&(ii)\; \lim\limits_{r\rightarrow 1}\sup\limits_{|\varphi_1(z)|>r}\left|\mathcal{T}_{\alpha-1}^{\varphi_1}(vu_1)(z)\right| \rho(z) +\lim\limits_{r\rightarrow 1}\sup\limits_{|\varphi_2(z)|>r}\left|\mathcal{T}_{\alpha-1}^{\varphi_2}(vu_2)(z)\right| \rho(z) \nonumber\\&&+ \lim\limits_{r\rightarrow 1}\sup\limits_{\min\{|\varphi_1(z)|,|\varphi_2(z)|\}>r}\left|\mathcal{T}_{\alpha-1}^{\varphi_1}(vu_1)(z)
-\mathcal{T}_{\alpha-1}^{\varphi_2}(vu_2)(z)\right|=0; \nonumber\\&&(iii)\; \limsup\limits_{|a|\rightarrow 1}\|(u_1C_{\varphi_1}-u_2C_{\varphi_2}) f_{a}\|_v+\limsup\limits_{|a|\rightarrow 1}\|(u_1C_{\varphi_1}-u_2C_{\varphi_2}) g_{a}\|_v =0;\nonumber\\&& (iv)\;\limsup\limits_{n\rightarrow \infty}  n^{\alpha-1}\|u_1 \varphi_1^{n} -u_2 \varphi_2^n \|_v=0.  \end{eqnarray*}\end{corollary}

\textbf{Case III}:\; Let $u_1=u_2=id$, the identity map,  and $\;m=1,$  it's trivial that $D_{\varphi_1,u_1}^m-D_{\varphi_2,u_2}^m=C_{\varphi_1}D-C_{\varphi_2}D.$  And in this case, we still use $f_a$ and $g_a$ as below.
\begin{eqnarray*}&& f_a(z)=\int_0^z\frac{(1-|a|^2)^\alpha}{(1-\bar{a}t)^{2\alpha}}dt, \\&&
  g_a(z)=\int_0^z  \frac{(1-|a|^2)^\alpha}{(1-\bar{a}z)^{2\alpha}}\cdot \frac{a-t}{1-\bar{a}t} dt. \end{eqnarray*}
\begin{corollary} Let  $0<\alpha<\infty$, $v$ be a weight  and $\varphi_1, \varphi_2 \in S(\mathbb{D})$. Then the following statements are equivalent,

$(i)$ $C_{\varphi_1}D-C_{\varphi_2}D: \mathcal{B}^\alpha\rightarrow H_v^\infty$ is bounded;

$(ii)$ \begin{eqnarray*}&&\sup\limits_{z\in \mathbb{D}}|\mathcal{T}_{\alpha}^{\varphi_1}(v)(z)|\rho(z)+\sup\limits_{z\in \mathbb{D}}\left|\mathcal{T}_{\alpha}^{\varphi_1}(v)(z) - \mathcal{T}_{\alpha}^{\varphi_2}(v)(z) \right| <\infty,  \\&&  \sup\limits_{z\in \mathbb{D}}|\mathcal{T}_{\alpha}^{\varphi_2}(v)(z)|\rho(z)+\sup\limits_{z\in \mathbb{D}}\left|\mathcal{T}_{\alpha}^{\varphi_1}(v)(z) - \mathcal{T}_{\alpha}^{\varphi_2}(v)(z) \right| <\infty;
\end{eqnarray*}

$(iii)$ \begin{eqnarray*} \sup\limits_{a\in \mathbb{D}}\|( C_{\varphi_1}D- C_{\varphi_2}D) f_a\|_v+\sup\limits_{a\in \mathbb{D}} \|(C_{\varphi_1}D- C_{\varphi_2}D) g_a\|_v <\infty;\end{eqnarray*}

$(iv)$ \begin{eqnarray*}\sup\limits_{n\in \mathbb{N}_0}  n^{\alpha }\|\varphi_1 ^{n}-\varphi_2^n \|_v<\infty.  \end{eqnarray*}
\end{corollary}

\begin{corollary} Let $0<\alpha<\infty$, $v$ be a weight and  $\varphi_1, \varphi_2 \in S(\mathbb{D})$. Suppose that $ C_{\varphi_i}D : \mathcal{B}^\alpha\rightarrow H_v^\infty$ is bounded for $i=1,2$, then the following statements are equivalent,
 \begin{eqnarray*}&&(i)\; \mbox{ $ C_{\varphi_1}D -  C_{\varphi_2}D: \mathcal{B}^\alpha\rightarrow H_v^\infty$ is compact;}\\&&(ii)\; \lim\limits_{r\rightarrow 1}\sup\limits_{|\varphi_1(z)|>r}\left|\mathcal{T}_{\alpha }^{\varphi_1}(v )(z)\right| \rho(z) +\lim\limits_{r\rightarrow 1}\sup\limits_{|\varphi_2(z)|>r}\left|\mathcal{T}_{\alpha}^{\varphi_2}(v )(z)\right| \rho(z) \nonumber\\&&+ \lim\limits_{r\rightarrow 1}\sup\limits_{\min\{|\varphi_1(z)|,|\varphi_2(z)|\}>r}\left|\mathcal{T}_{\alpha }^{\varphi_1}(v )(z)-\mathcal{T}_{\alpha }^{\varphi_2}
(v )(z)\right|=0; \nonumber\\&&(iii)\; \limsup\limits_{|a|\rightarrow 1}\|( C_{\varphi_1}D- C_{\varphi_2}D) f_{a}\|_v+\limsup\limits_{|a|\rightarrow 1}\|( C_{\varphi_1}D- C_{\varphi_2}D) g_{a}\|_v =0;\nonumber\\&& (iv)\;\limsup\limits_{n\rightarrow \infty}  n^{\alpha}\|  \varphi_1^{n} -  \varphi_2^n \|_v=0.  \end{eqnarray*}\end{corollary}

\textbf{Case IV}:  Let $m=1,$  it's trivial that $D_{\varphi_1,u_1}^m-D_{\varphi_2,u_2}^m=u_1 C_{\varphi_1}D-u_2 C_{\varphi_2}D.$  And in this case, we also use $f_a$ and $g_a$  to stand for the following test functions.
\begin{eqnarray*}&& f_a(z)=\int_0^z\frac{(1-|a|^2)^\alpha}{(1-\bar{a}t)^{2\alpha}}dt, \\&&
  g_a(z)=\int_0^z  \frac{(1-|a|^2)^\alpha}{(1-\bar{a}z)^{2\alpha}}\cdot \frac{a-t}{1-\bar{a}t} dt. \end{eqnarray*}
\begin{corollary} Let $0<\alpha<\infty$ and $v$ be a weight. Suppose $u_1, u_2\in H(\mathbb{D})$ and $\varphi_1, \varphi_2 \in S(\mathbb{D})$. Then the following statements are equivalent,

$(i)$ $u_1 C_{\varphi_1}D-u_2 C_{\varphi_2}D: \mathcal{B}^\alpha\rightarrow H_v^\infty$ is bounded;

$(ii)$ \begin{eqnarray*}&&\sup\limits_{z\in \mathbb{D}}|\mathcal{T}_{\alpha}^{\varphi_1}(vu_1)(z)|\rho(z)+\sup\limits_{z\in \mathbb{D}}\left|\mathcal{T}_{\alpha}^{\varphi_1}(vu_1)(z) - \mathcal{T}_{\alpha}^{\varphi_2}(vu_2)(z) \right| <\infty,\\&&  \sup\limits_{z\in \mathbb{D}}|\mathcal{T}_{\alpha}^{\varphi_2}(vu_2)(z)|\rho(z)+\sup\limits_{z\in \mathbb{D}}\left|\mathcal{T}_{\alpha}^{\varphi_1}(vu_1)(z) - \mathcal{T}_{\alpha}^{\varphi_2}(vu_2)(z) \right| <\infty;\end{eqnarray*}

$(iii)$ \begin{eqnarray*} \sup\limits_{a\in \mathbb{D}}\|(u_1 C_{\varphi_1}D-u_2 C_{\varphi_2}D)  f_a\|_v+\sup\limits_{a\in \mathbb{D}} \|(u_1 C_{\varphi_1}D-u_2 C_{\varphi_2}D) g_a\|_v <\infty; \end{eqnarray*}

$(iv)$ \begin{eqnarray*}\sup\limits_{n\in \mathbb{N}_0}  n^{\alpha}\|u_1\varphi_1^{n}-u_2\varphi_2^n\|_v<\infty.\end{eqnarray*}
\end{corollary}
\begin{corollary} Let $0<\alpha<\infty$, $v$ be a weight. Suppose $u_1, u_2\in H(\mathbb{D})$ and $\varphi_1, \varphi_2 \in S(\mathbb{D})$. Suppose that $u_i C_{\varphi_i}D: \mathcal{B}^\alpha\rightarrow H_v^\infty$ is bounded for $i=1,2$, then  the following statements are equivalent,   \begin{eqnarray*}&&(i)\; \mbox{$u_1 C_{\varphi_1}D-u_2 C_{\varphi_2}D: \mathcal{B}^\alpha\rightarrow H_v^\infty$ is compact;}\\ && (ii)\lim\limits_{r\rightarrow 1}\sup\limits_{|\varphi_1(z)|>r}\left|\mathcal{T}_{\alpha}^{\varphi_1}(vu_1)(z)\right| \rho(z) +\lim\limits_{r\rightarrow 1}\sup\limits_{|\varphi_2(z)|>r}\left|\mathcal{T}_{\alpha}^{\varphi_2}(vu_2)(z)\right| \rho(z) \nonumber\\&&+ \lim\limits_{r\rightarrow 1}\sup\limits_{\min\{|\varphi_1(z)|,|\varphi_2(z)|\}>r}\left|\mathcal{T}_{\alpha}^{\varphi_1}(vu_1)(z)-
\mathcal{T}_{\alpha}^{\varphi_2}
(vu_2)(z)\right|=0; \nonumber\\&&(iii)\; \limsup\limits_{|a|\rightarrow 1}\|(u_1 C_{\varphi_1}D-u_2 C_{\varphi_2}D) f_{a}\|_v+\limsup\limits_{|a|\rightarrow 1}\|(u_1 C_{\varphi_1}D-u_2 C_{\varphi_2}D) g_{a}\|_v =0;\nonumber\\&& (iv)\;\limsup\limits_{n\rightarrow \infty}  n^{\alpha}\|u_1 \varphi_1^{n} -u_2 \varphi_2^n \|_v=0.  \end{eqnarray*}\end{corollary}
%{\bf Acknowledgement.} The authors thank the editor and referees for many useful comments and
%suggestions. This paper was supported in part by the National Natural Science Foundation
%of China (Grant Nos. 10971153, 11126164, 11201331).

\bibliographystyle{amsplain}

\end{document}